\documentclass{amsart}

\usepackage[OT2,OT1]{fontenc}
\usepackage[utf8]{inputenc}
\usepackage{color}
\usepackage[colorlinks=true,linkcolor=blue,citecolor=blue]{hyperref}
\usepackage{tikz}
\usetikzlibrary{arrows,babel}
\usepackage{changes}
\usepackage{graphicx}
\usepackage{microtype}

\newtheoremstyle{theorem2}
{8pt}
{8pt}
{\itshape}
{}
{\bfseries}
{.}
{.5em}
{}

\newtheoremstyle{definition2}
{8pt}
{8pt}
{}
{}
{\bfseries}
{.}
{.5em}
{}

\theoremstyle{theorem2}

\newtheorem{lemma}{Lemma}[section]
\newtheorem{theorem}[lemma]{Theorem}
\newtheorem{corollary}[lemma]{Corollary}
\newtheorem{proposition}[lemma]{Proposition}
\DeclareMathOperator{\tast}{\tilde{\ast}}

\theoremstyle{definition2}

\newtheorem{remark}[lemma]{Remark}

\def\aa{{\mathfrak a}}

\def\bb{{\mathfrak b}}
\def\gg{{\mathfrak g}}
\def\hh{{\mathfrak h}}

\def\real{{\mathbb R}}

\def\HH{{\mathbb H}}

	\def\itemi{\item[{\it i)}]}
\def\itemii{\item[{\it ii)}]}

\title[The Brunn-Minkowski inequality in nilpotent Lie groups]{A direct proof of the Brunn-Minkowski inequality in Nilpotent Lie groups}
\author{Julián Pozuelo
}
\address{Departamento de Geometría y Topología, Facultad de Ciencias, Universidad de Granada, 18071 Granada, Spain}
\email{julipd94@correo.ugr.es}
\thanks{The author would like to thank his Ph.D. advisor Manuel Ritoré for sugesting the problem and his help.}

\numberwithin{equation}{section}

\begin{document}

	
\maketitle
		\thispagestyle{empty}
\begin{abstract}	
The purpose of this work is to give a direct proof of the multiplicative Brunn-Minkowski inequality in nilpotent Lie groups based on Hadwiger-Ohmann's one of the classical Brunn-Minkowski inequality in Euclidean space.
\end{abstract}
	
\section{Introduction}	

The classical Brunn-Minkowski inequality in Euclidean space asserts that, given $A,B\subset \real^d$ measurable sets, we have
\[
|A+B|^{1/d}\ge |A|^{1/d}+|B|^{1/d},
\]
where $|\cdot|$ indicates the volume of a set, and $A+B=\{a+b : a\in A, b\in B \}$ is the classical Minkowski addition of sets. Taking $\lambda\in[0,1]$, and replacing $A$ by $\lambda A$ and $B$ by $(1-\lambda)B$, we get the equivalent inequality
\[
|\lambda A+(1-\lambda)B|^{1/d}\geq\lambda|A|^{1/d}+(1-\lambda)|B|^{1/d}.
\]

There are several ways of generalizing the Brunn-Minkowski inequality. In Lie groups we can replace Minkowski addition of sets using the group product and take as volume the Haar measure of the group. This extension is called the \emph{multiplicative} Brunn-Minkowski inequality. In general metric measure spaces the notion of $s$-intermediate points can be used to replace the convex combination of points in Euclidean space, see \cite{MR3742604}. This leads to the \emph{geodesic} Brunn-Minkoski inequality.

In 2003, Monti \cite{MR1976833} observed that the multiplicative Brunn-Minkowski inequality in the sub-Riemannian Heisenberg group $\mathbb{H}^n$ cannot hold with exponent $(2n+2)^{-1}$, corresponding to the homogeneous dimension of $\mathbb{H}^n$, because otherwise Carnot–Ca\-ra\-théodory balls would be isoperimetric sets.

 However, Leonardi and Masnou proved in \cite{MR2177813} that this inequality holds 
with exponent $(2n+1)^{-1}$, corresponding to the topological dimension of $\mathbb{H}^n$.
Their proof was based on Hadwiger-Ohmann's proof of the classical Brunn-Minkowski inequality.

Later on, Tao \cite{tao,tao2} posted an entry in his blog explaining how to produce a Prékopa-Leindler inequality in any nilpotent Lie group of topological dimension $d$, which provides a natural way to prove the multiplicative Brunn-Minkowski inequality with exponent $d^{-1}$.

Juillet \cite{MR2520783} gave examples of sets for which the multiplicative Brunn-Minkowski inequality in $\mathbb{H}^n$ does not hold with exponent smaller than $(2n+1)^{-1}$.

In this article we give a proof  of the multiplicative Brunn-Minkowski inequality for nilpotent Lie groups following the approach by Leonardi and Masnou. To do it we use the special expression of the group product in exponential coordinates of the first kind and a generalization of the Brunn-Minkowski inequality in Euclidean space where the Minkowski content of sets is replaced using any product $\ast:\real^d\times\real^d\to\real^d$ of the form
\begin{equation*}
z\ast w=z+w+(F_1,F_2(z,w),\ldots, F_d(z,w))=z+w+F(z,w),
\end{equation*}
where $F_1$ is a constant and $F_i$ are continuous functions that depend only on $z_1,\ldots,z_{i-1},w_1,\ldots,w_{i-1}$ $\forall i=2,\ldots,d$. By a product here we mean a binary operation without assuming any further properties such as associativity. At the end of the paper, we state several classical variations of this inequality in the case of Carnot groups, where dilations can be defined.

\section{Preliminaries}
We recall some results on nilpotent and stratifiable groups. For a quite complete description of nilpotent Lie groups the reader is referred to \cite{MR1920389}, and to \cite{MR3742567} for stratifiable and Carnot groups.

	Let $\gg$ be a Lie algebra. We define recursively $\gg_0=\gg$, $\gg_{i+1}=[\gg,\gg_i]=\mbox{span}\{[X,Y] : X\in\gg,Y\in\gg_i \}$. The decreasing series		
	$$\gg=\gg_0\supseteq\gg_1\supseteq\gg_2\supseteq\ldots$$
	is called the \emph{lower central series of $\gg$}. If $\gg_r=0$ and $\gg_{r-1}\neq0$ for some $r$, we say that $\gg$ is \emph{nilpotent}, and the number $r$ is called the \emph{step of $G$}. A connected and simply connected Lie group is said to be \emph{nilpotent} if its Lie algebra is nilpotent.

Notice that each $\gg_i$ is an ideal in $\gg$. We shall write $n_i$ for the dimension of $\gg_i$.

\begin{lemma}\label{seq}
	Let $\gg$ be a nilpotent Lie algebra. Then there exists a basis $\{X_1,\ldots,X_d\}$ of $\gg$ such that 
	\begin{enumerate}
		\itemi for each $1\leq n\leq d$, $\hh_n=\mbox{span}\{X_{d-n+1},\ldots,X_d\}$ is an ideal of $\gg$,
		\itemii for each $0\leq i\leq r-1$, $\hh_{n_i}=\gg_i$.
	\end{enumerate}	
\end{lemma}

 A basis verifying this is called a \emph{strong Malcev basis}. This construction is adapted from \cite{MR1070979}.

Fixed a strong Malcev basis, the exponential is  a diffeomorphism between $\real^d$ and $G$, and is given by the map 
$$x=(x_1,\ldots,x_d)\mapsto \exp(x_1X_1+...+x_dX_d). $$
 This result can be found as Theorem 1.127 in \cite{MR1920389}. By abuse of notation we shall denote $\exp(x_1X_1+...+x_dX_d)=\exp(x)$, and $\exp_G$ if specifying the group is needed. The inverse of this map provides coordinates called \emph{canonical coordinates of the first kind}, and we denote it as $\log:G\rightarrow\real^d$.

%
We define a multiplication map associated to the exponential in a nilpotent group by
$$z\ast w=\log(\exp(z)\cdot\exp(w)).$$

The structure of this product is given by the following Theorem. It was first proved by Malcev in 1949 \cite{MR0028842}, and a proof can be found as Theorem 4.1 in \cite{MR1349140}, or with some modification as Proposition 1.2.7 in \cite{MR1070979}.
\begin{theorem}\label{multi}
	Let $G$ be a nilpotent group. Then the multiplication map takes the following form:
	\begin{equation}\label{1.130}
	z\ast w=z+w+(P_1(z,w),\ldots,P_d(z,w)),
	\end{equation}
	where $P_1$ is a constant and $P_i$ is a polynomial in $z_1,\ldots,z_{i-1},w_1,\ldots,w_{i-1}$ $\forall \ i=2,\ldots,d$.
%
\end{theorem}

We stop here to show that, slightly refining Theorem \ref{multi}, the multiplication map acts as a sum in the coordinates corresponding to the complement of $\gg_1$. 

\begin{theorem}\label{layer}
	Let $G$ be a nilpotent group. Then the multiplication map takes the following form:
	\begin{equation*}
	z\ast w=z+w+(0,\ldots,0,P_{d-n_1+1}(z,w),\ldots,P_d(z,w))
	\end{equation*}
		where $P_{d-n_1+1}$ is a constant and $ P_i$ is a polynomial in $z_1,\ldots,z_{i-1},w_1,\ldots,w_{i-1}$ $\forall \ i=d-n_1+2,\ldots,d$.
\end{theorem}
\begin{proof}
	Let $Z=\sum_{i=1}^d z_iX_i$, $W=\sum_{i=1}^d w_iX_i$. Since $\gg_1$ is an ideal in $\gg$, there is a normal Lie subgroup $G_1\subseteq G$ whose Lie algebra is $\gg_1$. Let $\pi:G\rightarrow G/G_1$, $\tilde{z}=\pi(z),$ $\tilde{w}=\pi(w)$, $\tilde{Z}=(d\pi)_0(Z)$, $\tilde{W}=(d\pi)_0(W)$. Notice that $\ker(d\pi)_0=\hh_{n_1}$ and $\gg/\gg_1$ is a trivial Lie algebra with the induced product. As a consequence of the Baker-Campbell-Hausdorff formula, $\tilde{z}\ast\tilde{w}=\tilde{z}+\tilde{w}$. On the other hand we calculate $\tilde{z}\ast\tilde{w}$ taking $\log_{G/G_1}$ in the equation below.
	\begin{equation*}
	\begin{split}
	&\exp_{G/G_1}(\tilde{Z})\exp_{G/G_1}(\tilde{W})=\pi(\exp_G(Z)\exp_G(W))=\\
	&\pi\big(\exp_G\big(Z+W+\sum_{i=1}^d P_i(z,w)X_i\big)\big)=\exp_{G/G_1}\big(\tilde{Z}+\tilde{W}+\sum_{i=1}^{n_1}P_i(z,w)X_i\big).
	\end{split}
	\end{equation*}
	Joining both expressions we obtain that $P_i=0 \ \forall i=1,\ldots, n_1$. 
\end{proof}	

From Theorem \ref{layer} it can be proved that left and right translations are maps whose Jacobian determinant is equal to $1$ at any point,
and the change of variables gives us the following Theorem. The interested reader can find the details as Theorem 1.2.9 and Theorem 1.2.10 in \cite{MR1070979}.

\begin{proposition}\label{mappreserving}
	Let $G$ be a nilpotent group. Then the exponential takes the Lebesgue measure on $\real^d$ to a Haar measure $\mu$ on $G$, that is,
	for any $A\subset G$ measurable and any $f:G\to\real$ integrable, it is satisfies
	\begin{equation*}
	\mu(A)=|\log(A)|\quad \text{and}  \quad \int_G f(\sigma)d\mu(\sigma)=\int_{\real^d}(f\circ\exp)(x)dx.
	\end{equation*}
\end{proposition}
%
	We refer the reader to \cite{MR3742567} for the details on the rest of this section.

	A \emph{stratification} of a Lie algebra $\gg$ is a direct-sum decomposition
	$$\gg=V_1\oplus\ldots \oplus V_r,$$
	for some integer $r\geq1$, where $V_r\neq\{0\}$, $[V_1,V_i]=V_{i+1}$ for all $i\in\{1,\ldots ,r\}$ and $V_{r+1}=\{0\}$. We say that a Lie algebra is \emph{stratifiable} if there exists a stratification on it. 
	 We say that a Lie group is \emph{stratifiable}  if it is connected and simply connected and its Lie algebra is stratifiable.

The following lemma assures that any stratifiable group is a nilpotent group.
\begin{lemma}\label{strat}
	If $ \gg=V_1\oplus\ldots\oplus V_r$ is a stratified Lie algebra, then
	$$\gg_{k-1}=V_k\oplus\ldots\oplus V_r.$$
	In particular, $\gg$ is a nilpotent Lie algebra of step $r$, and $\gg=V_1\oplus\gg_1$.
\end{lemma}

It is worth checking that Theorem \ref{layer} manifests that the multiplication map acts as a sum in the coordinates corresponding to $V_1$. The reader can find an example of a nilpotent group which is not stratifiable in \cite{MR3742567}.

\begin{proposition}\label{unicstrat}
	Let $\gg$ be a stratifiable Lie algebra with stratifications
	\begin{equation*}
	\gg=V_1\oplus\ldots\oplus V_r=W_1\oplus\ldots\oplus W_s.
	\end{equation*}
	Then $r=s$ and there exists a Lie algebra automorphism $A:\gg\rightarrow\gg$ such that $A(V_i)=W_i$ for $i=1,\ldots,r$.
\end{proposition}	

Proposition \ref{unicstrat} guarantees that for a stratifiable group $G$, the natural number
$$Q=\sum_{i=1}^r i\dim(V_i),$$
does not depends on the particular stratification. $Q$ is called the \emph{homogeneous dimension of $G$}.

	For $\lambda>0$ we define the \emph{dilation on $\gg$ of factor $\lambda$} as the unique linear map $\delta_\lambda:\gg\rightarrow\gg$ such that
	$$\delta_{\lambda}(X)=\lambda^t X \quad \ \forall X\in V_t \ \forall t\in\{1,\ldots ,r\}.$$
\begin{remark}
	Dilations $\delta_{\lambda} :\gg\rightarrow\gg$ are Lie algebra isomorphisms.
\end{remark}

The simply connection of $G$ certifies that there exists a unique Lie groups automorphism $\delta_\lambda:G\rightarrow G$ (denoted as the dilation on the Lie algebra) whose differential at $e$ is the dilation on $\gg$ of factor $\lambda$. This automorphism is called \emph{dilation on $G$ of factor $\lambda$}.

\begin{proposition}\label{strhaar}
	Let $G$ be a stratifiable group with Haar measure $\mu$ and let $\lambda>0$. Then
	\[ \int_G fd\mu=\lambda^Q\int_G (f\circ\delta_{\lambda}) d\mu,\]
	where $Q$ is the homogeneous dimension of $G$.
\end{proposition}

Let $G$ be a stratified group, with the stratification $\gg=V_1\oplus V_2\oplus\ldots\oplus V_r$, and fix a norm $\|\cdot \|$ on $V_1$. We can construct a distance $d$  homogeneous with respect to $\delta_{\lambda}$, that is, $d(\delta_\lambda(p),\delta_\lambda(q))=\lambda d(p,q)$. First we extend $V_1$ and $\|\cdot \|$ to a left-invariant subbundle $\Delta$ of the tangent bundle and a left-invariant norm on $\Delta$ by left translations:
\begin{equation*}
\left\lbrace \begin{split}
&&\Delta_\sigma=(dl_\sigma)_e V_1\ \ \forall\sigma\in G\\
&&\|(dl_\sigma)_e(v)\|=\|v\| \ \ \forall v\in V_1.
\end{split}\right. 
\end{equation*}

Where $l_\sigma:G\rightarrow G$ $l_\sigma(\tau)=\sigma\tau$. Now we define the \emph{Carnot-Caratheodory distance} or \emph{CC-distance} associated with $\Delta$ and $\|\cdot \|$ via  piecewise smooth paths $\gamma\in C^\infty_{pw}([0,1],G)$ as
	\[ d(p,q)=\inf\left\lbrace \int_0^1 \|\gamma'(t)\|dt : \gamma\in C^\infty_{pw}([0,1],G),\ \gamma(0)=p,\ \gamma(1)=q, \ \gamma'(t)\in\Delta \right\rbrace ,\]
	We call the data $(G,\delta_\lambda,\Delta,\|\cdot\|,d)$ a \emph{Carnot group} or, more explicitely, \emph{subFinsler Carnot
	group}. Usually, the term Carnot group is used when the norm comes from a scalar product, but in this paper we shall make no distinction.

\section{The Brunn-Minkowski inequality}	

We have seen that any nilpotent group is isomorphic to $\real^d$ with a product  of the form \eqref{1.130}. Now we prove the Brunn-Minkowski inequality for any product $\ast:\real^d\times\real^d\to\real^d$ of the form
\begin{equation}\label{esc}
z\ast w=z+w+(F_1,F_2(z,w),\ldots, F_d(z,w))=z+w+F(z,w),
\end{equation}
where $F_1$ is a constant and $F_i$ are continuous functions that depend only on $z_1,\ldots,z_{i-1},w_1,\ldots,w_{i-1}$ $\forall i=2,\ldots,d$. This product does not necessarily defines a group structure in $\real^d$.
Given such a map $F$ and $z'_1,w'_1\in\real$, we can define another product $\ast_{z'_1,w'_1}:\real^d\times\real^d\to\real^d$, by
\begin{equation*}
z\ast_{z'_1,w'_1}w=z+w+F((z_1',\tilde{z}),(w'_1,\tilde{w})),
\end{equation*}
 where $\tilde{z}=(z_2,\ldots,z_d)$, $\tilde{w}=(w_2,\ldots,w_d)$. 
 We define the map $F_{(z'_1,w'_1)}:\real^{d-1}\times\real^{d-1}\to\real^{d-1}$ by 
 \begin{equation}\label{fmin}
 F_{(z'_1,w'_1)}(\tilde{z},\tilde{w}):=(F_2,\ldots,F_d)((z'_1,\tilde{z}),(w'_1,\tilde{w})).
 \end{equation}
 Notice that  $F_i((z'_1,\tilde{z}),(w'_1,\tilde{w}))$ only depends on the first $i-2$ variables of $\tilde{z}$ and $\tilde{w}$ and so $F_2((z'_1,\tilde{z}),(w'_1,\tilde{w}))$ is constant. Thus the product $\tast:\real^{d-1}\times\real^{d-1}\to\real^{d-1}$ given by
  \begin{equation}\label{menprod}
 \tilde{z}\tast\tilde{w}=\tilde{z}+\tilde{w}+F_{(z'_1,w'_1)}(\tilde{z},\tilde{w}),
 \end{equation}
 has the form \eqref{esc}. Notice that the product $\tast$ depends on the choice of $z'_1,w'_1$.
 
\begin{lemma}\label{lem}
Let $\ast:\real^d\times\real^d\to\real^d$ be a product of the form \eqref{esc} and let $A,B\subset \real^d$ be $A=I\times\tilde{A}$, and $B=J\times\tilde{B}$, where $I,J$ are compact intervals in $\real$ and $\tilde{A},\tilde{B}\subset\real^{d-1}$ are measurable. Then 
\begin{equation}\label{lemma} 
	|A\ast B|\geq|I+J|\mathfrak{L}^{d-1}\left(\tilde{A}\tast\tilde{B} \right),
\end{equation}
where $\tast$ is the product described in \eqref{menprod} for certain $z_1'\in I$ and $w'_1\in J$. Moreover, if $F$ does not depends on $z_1, w_1$, then equality holds in \eqref{lemma}.
\end{lemma}

\begin{proof}
Let $I=[a,b]$, $J=[a',b']$ and $l=b-a$, $l'=b'-a'$. The product is
	\begin{equation*}
	A\ast B=\left\lbrace z+w+F(z,w) : z_1\in I,\ w_1\in J, \ \tilde{z}\in \tilde{A}, \ \tilde{w}\in\tilde{B} \right\rbrace .
	\end{equation*}	
	We define a diffeomorphism $\phi:\real^2\to\real^2$ by $(s,z)\mapsto(z,s-z)$. The inverse $\phi^{-1}(z,w)=(z+w,z)$ is a diffeomorphism between the sets $I\times J$ and $\{(s_1,z_1) : s_1\in I+J, \ z_1\in I\cap (s_1-J)=K(s_1)\}$. Recall that  $F_{(z_1,w_1)}$ is defined in \eqref{fmin}, the change of variables gives us
	\begin{multline*}
	A\ast B=\Big\{  (s_1,\tilde{z}+\tilde{w})+(F_1,F_{\phi(s_1,z_1)})(\tilde{z},\tilde{w}) :  s_1\in I+J,\\   \tilde{z}\in \tilde{A},  z_1\in K(s_1),   \tilde{w}\in\tilde{B}\Big\}.
	\end{multline*}
	Now we use Fubini's Theorem and we obtain
	\begin{equation}\label{fm}
	|A\ast B|=\int_{I+J}h(s_1)ds_1,
	\end{equation}
	where $h:I+J\rightarrow\real^+_0$ is the function
	\begin{equation*}
	h(s_1)=\mathfrak{L}^{d-1}\left(\{\tilde{p}\in\real^{d-1} : (s_1+F_1,\tilde{p})\in A\ast B \}\right)=\mathfrak{L}^{d-1}\Bigg(\bigcup_{z_1\in K(s_1)}D_{(s_1,z_1)} \Bigg),
	\end{equation*}
	and 
	\begin{equation}\label{de}
	D_{(s_1,z_1)}=\{\tilde{z}+\tilde{w}+F_{\phi(s_1,z_1)}(\tilde{z},\tilde{w}) : \ \tilde{z}\in \tilde{A}, \ \tilde{w}\in\tilde{B} \}.\\
	\end{equation}
	
	Now we compare the measure of $A\ast B$ with the measure of $D_{(s_1,z_1)}$ for some $s_1,z_1$.
 	Let $z_1:I+J\rightarrow\real$ be the function
	\[
	z_1(s_1)=\frac{s_1-(a+a')}{l+l'}l+a,
	\] 	
 	which is the composition of the inverse of the parametrization of $I+J$ given by $t\mapsto t(l+l')+(a+a')$ where $t$ is in $[0,1]$, with the parametrization of $I$, $t\mapsto tl+a'$. Since $\frac{s_1-(a+a')}{l+l'}\in[0,1]$, $z_1(s_1)\in I$. Besides
 	\begin{align*}
 	s_1-\frac{s_1-(a+a')}{l+l'}l-a&=\frac{l'(s_1-a)+a'l}{l+l'} \\
 	&=\frac{l'(t(l+l')+a')+a'l}{l+l'}=tl'+a',
 	\end{align*}
 	and therefore $s_1-z_1(s_1)\in J\Leftrightarrow z_1(s_1)\in s_1-J$. Then $z_1(s_1)\in K(s_1)$. 
 	 
 Let $f:I+J\rightarrow\real^+_0$ be the map given by $f(s_1)=\mathfrak{L}^{d-1}(D_{(s_1,z_1(s_1))})$. It is easy to check that $f$ is continuous, hence $f$ reaches its minimum at $s_1'$.
	\begin{equation}\label{des}	
	\int_{I+J}h(s_1)ds_1\geq\int_{I+J}f(s_1)ds_1\geq\int_{I+J}f(s'_1)ds_1=|I+J|f(s'_1).
	\end{equation}
	Denoting $z'_1:=z_1(s'_1), w'_1:=s'_1-z'_1$, we have that $F_{\phi(s'_1,z'_1)}=F_{(z'_1,w'_1)}$  and $D_{(s'_1,z'_1)}=\tilde{A}\tast\tilde{B}$. 
	
	If  $F$ does not depend on $z_1, w_1$, then $F_{\phi(s'_1,z'_1)}=F_{\phi(s_1,z_1)}$ for all $s_1$ and $z_1$, and therefore
\begin{equation*}
D_{(s'_1,z'_1)}=D_{(s_1,z_1(s_1))}=\bigcup_{z_1\in K(s_1)}D_{(s_1,z_1)}, 
\end{equation*}
this implies  $f(s'_1)=f(s_1)=h(s_1)\ \forall s_1\in I+J$, and the equality  holds in \eqref{des}.
\end{proof}
\begin{remark}
	Lemma \ref{lem} guarantees
	\begin{equation*}
	|A\ast B|\geq|I+J|\mathfrak{L}^{d-1}\left(\tilde{A}\tast\tilde{B}\right),
	\end{equation*}
	where $z'_1\in I$ $w'_1\in J$. Besides the product $\ast_{z_1',w_1'}$ does not depend on $z_1,w_1$ and therefore
	\begin{equation}\label{astast'}
	|A\ast B|\geq|I+J|\mathfrak{L}^{d-1}\left(\tilde{A}\tast\tilde{B}\right)=|A\ast_{z_1',w_1'}B|.
	\end{equation}	
	Recall that $\ast_{z_1',w_1'}$ acts as a sum in the first two coordinates, and someway \eqref{astast'} allows us to compare the measure of $A\ast B$ with the measure of a set more similar to the Euclidean Minkowski addition of $A$ and $B$.
\end{remark}	

\begin{theorem}[Brunn-Minkowski inequality for \eqref{esc} products]\label{prodbrunmin}
	Let $\ast:\real^d\times\real^d\to\real^d$ be a product of the form \eqref{esc} and let $A,B\subset \real^d$ be measurable sets such that $A\ast B$ is measurable. Then we have
	\begin{equation}\label{BM}
	|A \ast B|^{1/d}\geq |A|^{1/d}+|B|^{1/d}.
	\end{equation}
\end{theorem}

\begin{proof}
	The proof is divided into three steps.
	
	Step 1. We first claim that inequality holds in \eqref{BM} for a pair of $d$-rectangles $A$ and $B$, that is,
	\begin{equation*}
	\begin{split}
	A=&I_1\times\dots\times I_d \\
	B=&J_1\times\dots\times J_d.
	\end{split}
	\end{equation*}
	where $I_i,J_j$ are compact intervals $\forall\ 1\leq i,j\leq d$. 
	 We shall see that 
	\begin{equation}\label{asum}
	|A\ast B|\geq|I_1+J_1|\ldots|I_d+J_d|=|A + B|,
	\end{equation}
	and the classical Brunn-Minkowski inequality in $\real^d$ would imply \eqref{BM}.

	In order to prove \eqref{asum}, we use Lemma \ref{lem} to obtain
	\begin{equation*}
	|A\ast B|\geq|I_1+J_1|\mathfrak{L}^{d-1}\left(\tilde{A}\tast\tilde{B} \right)
	\end{equation*}
	but now $\tilde{A}=I_2\times(I_3\times\ldots\times I_d),\tilde{B}=J_2\times(J_3\times\ldots\times J_d)$ and $\tast$ has the form \eqref{esc} so we can apply Lemma \ref{lem} to the sets $\tilde{A}$ and $\tilde{B}$. Iterating this process, we get \eqref{asum}.

Step 2. Now we consider the case where $A$ and $B$ are finite unions of dyadic $d$-rectangles, that is, $A=A_1\cup\ldots\cup A_n$, $B=B_1\cup\ldots\cup B_m$ where $A_i=I^i_1\times\ldots\times I^i_d$, $B_j=J^j_1\times\ldots\times J^j_d$ and, for any $k=1,\ldots,d$ and $r\neq s$ ($p\neq q$), it is satisfied that $\text{int}(I^r_k)\cap\text{int} (I^s_k)=\emptyset$ or $I^r_k=I^s_k$ ($\text{int}(J^p_k)\cap \text{int}(J^q_k)=\emptyset$ or $J^p_k=J^q_k$), where $\text{int}(I)$ denotes the interior of $I$.

We reason by induction on the total number $n+m$ of $d$-rectangles. If $n+m=2$, then $A$ and $B$ are $d$-rectangles and we can apply step 1. Suppose that the Theorem holds for $n+m-1>2$. Then  we can find a hyperplane  $P: \{z_i=a_i \}$ such that some $A_r\subset\{z_i\geq a_i\}$ and some $A_s\subset\{z_i\leq a_i \}$. 

If the hyperplane has as equation $P: \{z_1=a_1 \}$, the proof is the same as the classical proof of Hadwiger and Ohmann for the sum of sets in $\real^d$. We include it for the sake of completeness. The sets $A^+=A\cap\{z_1\geq a_1 \}$ and $A^-=A\cap \{z_1\leq a_1 \}$ are unions of $d$-rectangles whose sum is strictly less than $n$. We choose a parallel hyperplane $Q:\{z_1=b_1\}$ verifying that, setting $B^+=\{ z_1\geq b_1\}$ and $B^-=\{z_1\leq b_1 \}$, then
\begin{equation}\label{vol1}
 \frac{|B^{\pm}|}{|B|}=\frac{|A^{\pm}|}{|A|}.
\end{equation}
Besides $B^+$ and $B^-$ are disjoint unions of $d$-rectangles whose sum is at most $m$.
 We apply the induction hypothesis to the pairs $A^+,B^+$ and $A^-,B^-$, and we obtain
\begin{equation}\label{vol2}
\begin{split}
|A^+\ast B^+|&\geq(|A^+|^{1/d}+|B^+|^{1/d})^d \\
|A^-\ast B^-|&\geq(|A^-|^{1/d}+|B^-|^{1/d})^d.
\end{split}
\end{equation}
On the other hand, $P\ast Q: \{z_1=a_1+b_1\}$ is another vertical plane in $\real^d$,  $A^+\ast B^+\subset (P\ast Q)^+$, and $A^-\ast B^-\subset (P\ast Q)^-$. Therefore $A^+\ast B^+$ and $A^-\ast B^-$ are disjoint sets (up to a null set) in $A \ast B$. Combining this with \eqref{vol1} and \eqref{vol2} we get the inequality
\begin{equation*}
\begin{split}
|A\ast B|&\geq|A^+\ast B^+|+|A^-\ast B^-|\\
&\geq (|A^+|^{1/d}+|B^+|^{1/d})^d+(|A^-|^{1/d}+|B^-|^{1/d})^d \\
&=(|A^+|+|A^-|)\left[1+ \left(\frac{|B|}{|A|} \right)^{1/d}\right]^d\\
&=(|A|^{1/d}+|B|^{1/d})^d,
\end{split}
\end{equation*}
and the Theorem is proved for such $A$ and $B$.

If there is no such hyperplane with equation $P: \{z_1=a_1 \}$ but with equation $P: \{z_2=a_2 \}$, then for any $u,v,p,q$,  $I^u_1= I^v_1=I_1$, $J^p_1= J^q_1=J_1$ and for some $r\neq s$, $\text{int}(I^r_2)\cap \text{int}(I^s_2)=\emptyset$, and we can write
\begin{equation*}
\begin{split}
A=&\bigcup_i I_1\times I^i_{2}\times\ldots\times I^i_d=I_1\times \left(\bigcup_i I^i_{2}\times\ldots\times I^i_d\right)=I_1\times\tilde{A}\\
B=&\bigcup_j J_1\times J^j_{2}\times\ldots\times J^j_d=J_1\times \left(\bigcup_j J_{2}^j\times \ldots\times J^j_d\right)=J_1\times \tilde{B}.
\end{split}
\end{equation*}
We have seen in \eqref{astast'} that
\[
|A\ast B|\geq|A\ast_{z_1',w_1'}B|.
\]
 Now we repeat the above argument, where now we apply the induction hypothesis to the product $\ast_{z_1',w_1'}$, thus the sets $A^{+}\ast_{z_1',w_1'}B^+$ and $A^-\ast_{z_1',w_1'}B^-$ are disjoint (up to a null set). Thus by \eqref{astast'} we obtain
\begin{equation*}
|A\ast B|\geq|A\ast_{z_1',w_1'}B|\geq|A^+\ast_{z_1',w_1'}B^+|+|A^-\ast_{z_1',w_1'}B^-|\geq(|A|^{1/d}+|B|^{1/d}|)^{1/d}.
\end{equation*}
and the result is proved. 

Repeating this reasoning we have covered the general case where $P: \{z_i=a_i \}$.

Step 3. Assume now that $A$ and $B$ are measurable sets such that $A\ast B$ is measurable. We can suppose that $A,B$ and $A\ast B$ have finite measure, since otherwise the inequality is trivial. Fix $\varepsilon>0$ and take an open set $O$ such that $A\ast B\subset O$ and $|O\setminus A\ast B|<\varepsilon$. Since $\ast$ is continuous, there exist open sets $O_A\supset A$ and $O_B\supset B$ such that $|O_{A}\setminus A|<\varepsilon$, $|O_{B}\setminus B|<\varepsilon$ and $O_A\ast O_B\subset O$. Now we approximate the sets $O_A$ and $O_B$ from inside by dyadic $d$-rectangles, $D_A$ and $D_B$ so that $|O_A \setminus D_A|<\varepsilon$, $|O_B \setminus D_B|<\varepsilon$. We use step 2 for $D_A$ and $D_B$. Taking $\varepsilon\rightarrow 0$ gives \eqref{BMN}.
\end{proof}

As a particular case, we have the Brunn-Minkowski inequality in nilpotent groups.

\begin{theorem}[Brunn-Minkowski inequality in nilpotent groups]\label{nilbrunnmin}
	Let $G$ be a  nilpotent group of topological dimension $d$ and let $A,B\subset G$ be measurable sets such that $A\cdot B$ is measurable. Then we have
	\begin{equation}\label{BMN}
	|A\cdot B|^{1/d}\geq |A|^{1/d}+|B|^{1/d}.
	\end{equation}
\end{theorem}

\begin{proof}
	 We denote $\aa=\log(A)$, $\bb=\log(B)$. Using Theorem~\ref{layer},  Proposition \ref{mappreserving} and Theorem \ref{prodbrunmin}, we have
	\begin{equation*}
	\begin{split}
	|A\cdot B|&=|\log(A\cdot B)|=|\log(\exp(\aa)\cdot\exp(\bb))|=|\aa \ast \bb|\geq(|\aa|^{1/d}+|\bb|^{1/d})^d \\
	&=(|A|^{1/d}+|B|^{1/d})^d. \qedhere
	\end{split}
	\end{equation*}
\end{proof}

\begin{remark}
In \cite{MR2177813}, Leonardi and Masnou considered only the hyperplanes $P:\{z_i=a_i\}$ on those coordinates where the product acts as a sum, the first $2n$ ones. Then for an open set $O\subseteq\HH^{n}\equiv\real^{2n+1}$, they consider the dyadic approximation and join all the cubes with the same projection on the first $2n$ coordinates. As a result, they obtain the generalized cube and use step 2. However, If we do this for general $O\subseteq\real^d$, projecting on the first $n_1$ coordinates, corresponding to the first layer where we have seen that the product acts as a sum in Theorem \ref{layer}, the union of the cubes takes the form 
	\[
	\bigcup_i I_1\times\ldots\times I_{n_1}\times I^i_{n_1+1} \times\ldots\times I^i_d=I_1\times\ldots\times I_{n_1}\times \left(\bigcup_i I^i_{n_1+1}\times\ldots\times I^i_d\right).
	\]
	This set is not usually a generalized cube, and we are not allowed to use step 2.
\end{remark}	

\begin{remark}
	Since the right-hand side of \eqref{BMN} is symmetric in $A$ and $B$, and it follows
	\[ \min\{|A\cdot B|, |B\cdot A| \}^{1/d}\geq |A|^{1/d}+|B|^{1/d}.  \]
	An example where $|A\cdot B|$ and $|B\cdot A|$ are different can be found in \cite{MR2177813}.
\end{remark}


\subsection{A sufficient condition for strict inequality in the Heisenberg group}
A set $A$ in the Heisenberg group $\HH^n$  of the form $A=A_1\times A_2$, where $A_1$ is a measurable set in $\real^{2n}$ and $A_2$ is a measurable set in $\real$ is called a \emph{generalized cylinder}.

In this subsection we prove in Lemma \ref{mayor} that the Brunn-Minkowski inequality \eqref{BMN} is strict in the Heisenberg group  for a pair of generalized cylinders $A$ and $B$ such that the volumes of $A_1$ and $B_1$ are positive. 

Recall that a point $a$ in $\real^d$ is a \emph{density point of $A$} if 
\[
\lim_{r\to0}\frac{|A\cap B(a,r)|}{|B(a,r)|}=  1,
\]
where $B(a,r)$ is the Euclidean ball of center $a$ and radius $r$. The set of density points of a set $A$ will be denoted as $A^o$. We can always normalize a set by including its density points in the set. The existence of a density point in $A$ 
implies that the volume of $A$ is positive.
\begin{proposition}\label{mayor}
Let $A,B\subset\HH^n$ be generalized cylinders. Suppose that $|A_1|>0$, $|B_1|>0$. Then 
\begin{equation}\label{mayorBM}
|A\cdot B|>|A+B|.
\end{equation}
\end{proposition}
\begin{proof}
	By Theorem \ref{layer} and Fubini's Theorem, we have
	\[
	|A\cdot B|=\int_{A_1+B_1}h(s_1)ds_1,
	\] 
	where $h(s_1)=|\{ t+t'+ \textrm{Im}(z\overline{(s_1-z)}) : t\in A_2, \ t'\in B_2, \ z\in K(s_1) \}|$. Denoting $s_1=(s_x,s_y)$ , we can see that $\textrm{Im}(z\overline{(s_1-z)})=\textrm{Im}(z\overline{s_1})=ys_x-xs_y$. 
	
	We assert that if $|K(s_1)|>0$, then $|\{\textrm{Im}(z\bar{s_1}) : z\in K(s_1)\}|>0$.
	To see that, we can take the diffeomorphism $\phi:\real^2\to\real^2$ given by $(x,y)\mapsto\big(ys_x-xs_y,\frac{x}{2s_x}-\frac{y}{2s_y}\big)$.
	Then  $|\textrm{Jac}(\phi)|=1$ and applying the change of variables formula to $\phi^{-1}$, we have 
	\[
	0<|K(s_1)|=\int_{\real^2}\chi_{K(s_1)}(z)dz=\int_{\real^2}\chi_{\phi(K(s_1))}(z)dz=|\phi(K(s_1))|.
	\]
	But if $|\phi(K(s_1))|>0$ then $|\pi_1(\phi(K(s_1)))|>0$ where $\pi_1(x,y)=x$, since for any set $O$, $|\pi_1(O)|=0$ implies $|O|\leq|\pi_1(O)\times\real|=0$. Hence $$|\{ys_x-xs_y : (x,y)\in K(s_1) \}|=|\pi_1(\phi(K(s_1)))|>0.$$
	 Let $I_{s_1}=\{ys_x-xs_y : (x,y)\in K(s_1) \}$. By the Brunn-Minkowski inequality in $\real$,
	\begin{multline*}
	h(s_1)=|\{ s_2+\textrm{Im}(z\overline{s_1}) : s_2\in A_2+B_2, \ z\in K(s_1) \}|\\
	\geq|\{ s_2+a : s_2\in A_2+B_2, \ a\in I_{s_1} \}|\geq |A_2+B_2|+|I_{s_1}|.
	\end{multline*}

	To complete the proof it remains to show that  $\{s_1\in A_1+B_1 : |K(s_1)|>0\}$ has positive measure.  Let $a\in  A_1^o$, $b\in  B_1^o$ and $s_1=a+b\in A_1^o+B_1^o$. Then  $a=s_1-b$ is a density point in $s_1-B_1$ 
	and therefore $a$ is a density point in $A_1\cap(s_1-B_1)=K(s_1)$ which implies that $|K(s_1)|>0$. Finally $A_1^o+B_1^o\subseteq A_1+B_1$ has positive measure since  $|A_1^o+B_1^o|\geq|A_1^o|=|A_1|>0$, and $|\{s_1\in A_1+B_1 : |K(s_1)|>0\}|\geq|\{s_1\in A_1^o+B_1^o : |K(s_1)|>0\}|>0$.
\end{proof}	
\begin{remark}
	In order to characterize the equality in \eqref{BMN} for generalized cylinders, we can distinguish several cases. If $A$ and $B$ lie in parallel vertical hyperplanes, then $|A\cdot B|=0$ and we have the equality in \eqref{BMN}. If $A$ and $B$  are convex and homothetic then either $|A_1|>0$ and $B_1$ is a point and the equality holds, or $|A_1|>0$ and  $|B_1|>0$, and by Lemma \ref{mayor} the equality \eqref{mayorBM} does not hold, and therefore by the Euclidean Brunn-Minkowski inequality the equality does not hold in \eqref{BMN}. The same argument works if $A$ and $B$ lie in horizontal hyperplanes with $|A_1|>0$ and $|B_1|>0$. The case in which $A$ and $B$ lie in horizontal hyperplanes with  $|A_1|=0$ is not known in general. 
\end{remark}	

\section{Consequences}
	Another equivalent version of the Brunn-Minkowski inequality in Euclidean space is the Prékopa-Leindler inequality. Now we show how the proof
	of the Prékopa-Leindler inequality from the Brunn-Minkowski inequality can be adapted to the case of nilpotent groups.
	
\begin{theorem}[Prékopa-Leindler inequality in nilpotent groups]\label{prelei}
		Let $G$ be a nilpotent group of topological dimension $d$ with Haar measure $\mu$. Let $f,g,h:G\rightarrow\real^{+}_0$ be measurable functions and $0<\alpha<1$ verifying
	\begin{equation}\label{PL} h(a\cdot b)\geq f(a)^{1-\alpha}g(b)^\alpha \ \ \forall a,b\in G.\end{equation}
	Then
	\begin{equation}\label{pre-lei} 
	\int_G h d\mu \geq \frac{1}{(1-\alpha)^{d(1-\alpha)}\alpha^{d\alpha}}\left(\int_G fd\mu\right)^{1-\alpha}\left(\int_G g d\mu\right)^{\alpha}. 	
	\end{equation}
\end{theorem}	

\begin{proof}
%
	We reason by induction on $d$. 
	
	Let $d=1$ and $a\cdot b\in \{f>\lambda \}\cdot\{g>\lambda \}$. Then we have $h(a\cdot b)\geq f(a)^{1-\alpha}g(b)^{\alpha}>\lambda,$ and as a consequence
	\[
	\{h>\lambda\}\supset\{f>\lambda \}\cdot\{g>\lambda \}.
	\]
	Now we can apply Theorem \ref{nilbrunnmin},
	\[
	\mu(\{h>\lambda\})\geq\mu(\{f>\lambda \})+\mu(\{g>\lambda \}).
	\]
	Integrating in $\lambda$ and using Cavalieri's Principle,
	\begin{multline}\label{1}
	\int_G hd\mu=\int_0^\infty \mu(\{h>\lambda\})d\lambda\geq\\
	 \int_0^\infty\big( \mu(\{f>\lambda \})+ \mu(\{g>\lambda \})\big) d\lambda=\int_G fd\mu+\int_G gd\mu.
	\end{multline}
	Now we use the weighted inequality between the geometric and arithmetic means,
	\begin{equation}\label{3}
	\int_G fd\mu+\int_G gd\mu\geq\left( \frac{\int_G fd\mu}{1-\alpha}\right)^{1-\alpha}\left(\frac{\int_G fd\mu}{\alpha}\right)^{\alpha}.
	\end{equation}
	From the equations \eqref{1} and \eqref{3} we have \eqref{pre-lei}.
	
	Suppose that Theorem \ref{prelei} holds for $d-1$. We shall prove the inequality \eqref{3} for the functions $f,g,h$ composed with $\exp$ and use Proposition \ref{mappreserving}. Let $z'=(z_1, \ldots,z_{d-1})$, $w'=(w_1,\ldots,w_{d-1})\in\real^{d-1}$. 
	 By \eqref{1.130}, we can write $(z',z_d)\ast (w',w_d)=(z'\ast'w',z_d+w_d+P_d(z',w'))$. Recall that $\real^d$ is isomorphic to $\gg$ once we fix the strong Malcev basis $\{X_1,\ldots,X_d\}$, and $X_d$ spans an ideal $\hh_1$ in $\gg$. Thus $\gg/\hh_1\cong(\real^{d-1},\ast')$ is a nilpotent group. Now we define the functions $\tilde{f},\tilde{g},\tilde{h}:\real\rightarrow \real^+_ 0$ by 
		 \begin{equation*}
	 \begin{split}
	 &\tilde{f}(z_d)=(f\circ\exp)(z',z_d),\\
	&\tilde{g}(w_d)=(g\circ\exp)(w',w_d), \\
	&\tilde{h}(t)=(h\circ\exp)(z'\ast' w',t+P_d(z',w')).
	 \end{split}
	 \end{equation*} 
	 Let us see that these functions verify \eqref{PL}:
	\begin{multline}\label{sos}
	\tilde{h}(z_d+w_d)=(h\circ\exp)((z',z_d)\ast (w',w_d))=h(\exp(z',z_d)\cdot\exp(w',w_d))\\
	\geq (f\circ\exp)(z',z_d)(g\circ\exp)(w',w_d)=\tilde{f}(z_d)\tilde{g}(w_d).
	\end{multline}
	 By induction hypothesis,
	 \begin{equation}\label{tildes}
	 \int_{\real} \tilde{h}(t)dt\geq \frac{1}{(1-\alpha)^{(1-\alpha)}\alpha^{\alpha}}\left(\int_{\real} \tilde{f}(z_d)dz_d\right)^{1-\alpha}\left(\int_{\real} \tilde{g}(w_d) dw_d\right)^{\alpha}. 	
	 \end{equation}
	By the invariance of the $1$-dimensional Lebesgue measure by translations we get
	 \begin{equation}\label{otramas}
	 \int_{\real}(h\circ\exp)(z'\ast' w',t)dt=\int_{\real}\tilde{h}(t)dt.
	 \end{equation}
	 The inequality \eqref{sos} is valid for any $z',w'\in\real^{d-1}$, and we can define the functions $F,G,H:\real^{d-1}\rightarrow\real^{+}_0$
	 \begin{equation}\label{HFG}
	 \begin{split}
	 &F(z')=\frac{1}{(1-\alpha)}\int_{\real} \tilde{f}(z_d)dz_d\\
	&G(w')=\frac{1}{\alpha}\int_{\real} \tilde{g}(w_d)dw_d\\
	&H(z')=\int_{\real} (h\circ\exp)(z',t)dt.
	 \end{split}
	 \end{equation}
	 Applying \eqref{otramas} we can rewrite \eqref{tildes} as
	 \begin{equation*}
	 H(z'\ast' w')=\int_{\real}\tilde{h}(t)dt\geq F(z')^{1-\alpha}G(w')^{\alpha} \ \ \forall z',w'\in\real^{d-1},
	 \end{equation*}
	 and again by the induction hypothesis, we get	 
	\begin{multline*}
	\int_{\real^{d-1}} H(z')dz'\geq\\
	\frac{1}{(1-\alpha)^{(d-1)(1-\alpha)}\alpha^{(d-1)\alpha}}\left(  \int_{\real^{d-1}}F(z')dz'\right)^{1-\alpha} \left(     \int_{\real^{d-1}}G(w')dw'\right) ^{\alpha}.
	\end{multline*}
	The Theorem follows from Fubini's Theorem and Proposition \ref{mappreserving}.	
\end{proof}

The Prékopa-Leindler inequality in $\real^d$ is usually stated using $h((1-\alpha)x+\alpha y)$ instead of $h(x+y)$ in order to eliminate the factor $((1-\alpha)^{d(1-\alpha)}\alpha^{d\alpha})^{-1}.$ This can be done when dilations are defined, and in this case, this inequality take a more pleasant expression.

\begin{corollary}\label{preleistrat}
	Let $G$ be a stratifiable group with topological dimension $d$ and homogeneous dimension $Q$.
	Let $f,g,h:G\rightarrow\real^{+}_0$ be measurable functions, and $0<\alpha<1$ verifying
	\begin{equation*} h(\delta_{(1-\alpha)}a\cdot\delta_{\alpha} b)\geq f(a)^{1-\alpha}g(b)^\alpha \ \ \forall a,b\in G.\end{equation*}
	Then 
	\begin{equation*}
	\int_G\,h d\mu \geq (1-\alpha)^{(Q-d)(1-\alpha)}\alpha^{(Q-d)\alpha} \left(\int_G fd\mu\right)^{1-\alpha}\left(\int_G g d\mu\right)^{\alpha}. 	
	\end{equation*}
\end{corollary}
\begin{proof}
	We denote  $a'=(1-\alpha) a$, $b'=\alpha b$, $f_{1-\alpha}(a)=f(\frac{a}{1-\alpha})$ and $g_\alpha(a)=g(\frac{a}{\alpha})$. Then we have
	$$h(a'\cdot b')\geq f(a)^{1-\alpha}g(b)^\alpha=f\left(\frac{a'}{1-\alpha}\right)^{1-\alpha}g\left(\frac{b'}{\alpha}\right)^{\alpha}=f_{1-\alpha}(a')^{1-\alpha} g_{\alpha}(b')^\alpha.$$
	By Theorem \ref{pre-lei}, we have
	$$\int_G hd\mu\geq\frac{1}{(1-\alpha)^{d(1-\alpha)}\alpha^{d\alpha}}\left(\int_G f_{1-\alpha}d\mu\right)^{1-\alpha}\left(\int_G g_\alpha d\mu\right)^{\alpha}.$$
	Using now Proposition \ref{strhaar},
	$$\int_G f_{1-\alpha}(a)d\mu(a)=\int_G f\left(\frac{a}{1-\alpha}\right)d\mu(a)=(1-\alpha)^Q\int_G f(a')d\mu(a'),$$
	and after using Proposition \ref{pre-lei} for the integral of $g_\alpha$, we obtain
	\[\int_G hd\mu\geq(1-\alpha)^{(Q-d)(1-\alpha)}\alpha^{(Q-d)\alpha}\left(\int_G fd\mu\right)^{1-\alpha}\left(\int_G gd\mu\right)^{\alpha}.\qedhere\]
\end{proof}

As we can find in \cite{MR3155183}, there are several equivalent statements for the Brunn-Minkowski inequality in Euclidean space.
\begin{corollary}[Multiplicative Brunn-Minkowski inequalities in Carnot groups]
	Let $G$ be a Carnot group with topological dimension $d$ and homogeneous dimension $Q$. Let $A,B\subset G$ be measurable sets such that $A\cdot B$ is measurable, and $0<\alpha<1$. Then
	\begin{equation*}
	\begin{split}
	&|\delta_{(1-\alpha)}A\cdot\delta_{\alpha}B|^{1/d}\geq (1-\alpha)^{Q/d}|A|^{1/d}+\alpha^{Q/d}|B|^{1/d}. \\
	&|\delta_{(1-\alpha)}A\cdot \delta_{\alpha} B|\geq (1-\alpha)^{(Q-d)(1-\alpha)}\alpha^{(Q-d)\alpha}|A|^{1-\alpha}|B|^{\alpha}.
	\end{split}
	\end{equation*}
\end{corollary}	
\begin{proof}
	We use Theorem \ref{nilbrunnmin} with the sets 	$\delta_{(1-\alpha)}A$ and $\delta_{\alpha}B$, and from Proposition \ref{strhaar} we get the first inequality.
	
	For the second one, we take $f=\chi_A$, $g=\chi_B$ and $h=\chi_{\delta_{(1-\alpha)}A\cdot\delta_{\alpha} B}$ and apply Corollary \ref{preleistrat}, obtaining the result. 
\end{proof}
%
%

	\bibliographystyle{abbrv}

\bibliography{Brunn-Minkowski}

\end{document}